\documentclass{article}

\usepackage{arxiv}

\usepackage[utf8]{inputenc} 
\usepackage[T1]{fontenc}    
\usepackage[hidelinks]{hyperref}      
\usepackage{url}            
\usepackage{booktabs}       
\usepackage{amsfonts}       
\usepackage{nicefrac}       
\usepackage{microtype}      
\usepackage{lipsum}		
\usepackage{graphicx}
\usepackage{amsmath}
\usepackage{amssymb}
\usepackage{amsthm}
\usepackage{enumerate}
\usepackage{multicol}
\usepackage{bm}         
\usepackage{flexisym}
\usepackage{enumitem}
\usepackage{natbib}
\usepackage{doi}
\usepackage{acro}
\usepackage{accents}
\usepackage{dsfont}

\newtheorem{theorem}{Theorem}
\newtheorem{assumption}{Assumption}
\newtheorem{lemma}{Lemma}
\newtheorem{definition}{Definition}
\newtheorem{remark}{Remark}

\newcommand{\ubar}[1]{\underaccent{\bar}{#1}}

\DeclareMathOperator{\diag}{diag}
\DeclareMathOperator{\rank}{rank}

\DeclareMathOperator{\col}{col}

\newcommand{\R}{\mathbb{R}}

\newcommand{\an}{{\alpha}}
\newcommand{\ab}{{\bar\alpha}}

\newcommand{\opt}[1]{\mathrm{MIN}(#1)}

\DeclareAcronym{mip}{
	short = MIP ,
	long  = Mixed-Integer Programming ,
	sort  = M ,
}
\DeclareAcronym{der}{
	short = DER ,
	long  = Distributed Energy Resources ,
	sort  = D ,
}
\DeclareAcronym{rtm}{
	short = RTM ,
	long  = Real-Time Market ,
	sort  = R ,
}
\DeclareAcronym{ads}{
	short = ADS ,
	long  = Active Demand and Supply ,
	sort  = A ,
}
\DeclareAcronym{brp}{
	short = BRP ,
	long  = Balance Responsible Party ,
	sort  = B ,
}
\DeclareAcronym{tso}{
	short = TSO ,
	long  = Transmission System Operator ,
	sort  = T ,
}
\DeclareAcronym{hp}{
	short = HP ,
	long  = Heat Pump ,
	sort  = H ,
}
\DeclareAcronym{mchp}{
	short = mCHP ,
	long  = Micro Combined Heat and Power  ,
	sort  = m ,
}
\DeclareAcronym{mpec}{
	short = MPEC ,
	long  = Mathematical Programming with Equilibrium Constraints  ,
	sort  = M ,
}

\DeclareAcronym{res}{
	short = RESs ,
	long  = Renewable Energy  Sources,
	sort  = R,
}
\DeclareAcronym{ict}{
	short = ICT ,
	long  = Information and Communications Technology,
	sort  = I,
}

\title{Bilevel Aggregator-Prosumers' Optimization Problem in Real-Time: A Convex Optimization Approach}

\author{ \href{https://orcid.org/0000-0003-4493-6495}{\includegraphics[scale=0.06]{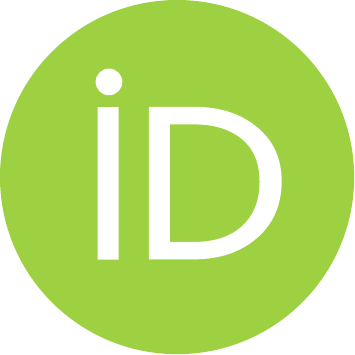}\hspace{1mm}Koorosh Shomalzadeh} \\
	Bernoulli Institute for
	Mathematics, Computer Science and Artificial Intelligence\\
	Faculty of Science and Engineering, University of Groningen\\
	Nijenborgh 9, 9747 AG, Groningen, The Netherlands \\
	\texttt{k.shomalzadeh@rug.nl} \\
	\And
	\href{https://orcid.org/0000-0002-3409-5760}{\includegraphics[scale=0.06]{orcid.pdf}\hspace{1mm}Jacquelien M.~A.~Scherpen} \\
	Engineering and Technology Institute Groningen\\
	Faculty of Science and Engineering, University of Groningen\\
	Nijenborgh 4, 9747 AG, Groningen, The Netherlands \\
	\texttt{j.m.a.scherpen@rug.nl} \\
	\AND
	\href{https://orcid.org/0000-0002-2407-8166}{\includegraphics[scale=0.06]{orcid.pdf}\hspace{1mm}M.~Kanat Camlibel} \\
	Bernoulli Institute for
	Mathematics, Computer Science and Artificial Intelligence\\
	Faculty of Science and Engineering, University of Groningen\\
	Nijenborgh 9, 9747 AG, Groningen, The Netherlands \\
	\texttt{m.k.camlibel@rug.nl} \\
}

\date{}

\renewcommand{\headeright}{}
\renewcommand{\undertitle}{}
\renewcommand{\shorttitle}{Bilevel Aggregator-Prosumers' Optimization Problem: A Convex  Approach}

\begin{document}
\maketitle

\begin{abstract}
This paper proposes a \ac*{rtm} platform for an aggregator and its corresponding prosumers to participate in the electricity wholesale market. The proposed energy market platform is modeled as a bilevel optimization problem where the aggregator and the prosumers are considered as self-interest agents. The current state-of-the-art \ac*{mpec} and \ac*{mip} based approaches to solve bilevel optimization problems are not satisfactory for real-time applications. The computation time for such approaches grows exponentially as the number of prosumers and decision variables increase. This paper presents a convex optimization problem which can capture a subset of the set of global optima of the bilevel problem as its optimal solution. 
\end{abstract}


	\section{Introduction}
Power systems are experiencing a fundamental transition. Previously, the energy was generated in the bulk power plants and it was flowing through transmission and distribution networks to the consumers. The massive installation of \ac{res} at the household level has challenged this structure. Therefore, new schemes and models are needed to efficiently cope with this transition \citep{bollen2011integration}. 
\par
The emergence of the energy producing consumers, i.e., prosumers 
and recent \ac{ict} developments in the paradigm of smart grid \citep{gungor2011smart} have opened up new horizons for  less grid-dependent households.    
Since output generation of  \ac{res} are volatile due to their intrinsic environmental dependency, researchers have proposed different approaches to address demand and supply matching for a group of prosumers.
Utilization of storage devices \citep{roberts2011role}, bilateral energy transactions between prosumers \citep{bedoya2019bilateral}, and  bilateral energy transaction between prosumers and the wholesale market \citep{zugno2013bilevel} are among the most prominent of those approaches. 
\par
Here in this paper we focus on a real-time grid-prosumers energy transaction through an aggregator as the mechanism to address demand and supply  matching. 
The aggregator's role is to gather and manage  a group of prosumers in order to participate in the real-time wholesale market.
\par 
Many types of aggregator with dissimilar goals have been studied in different financial and market structures in the area of electricity markets \citep{martin2016literature}. In this work, the aggregator is a self-interest market participant who has the goal of participating in the  real-time wholesale market in order to maximize its revenue.  To do so, the aggregator considers each individual prosumer demand and supply situation and proposes a personalized price  to buy its excess supply or provide the prosumer its energy deficiency at each time-step in a \ac{rtm}.
\par
On the other hand, each prosumer receives a price from the aggregator and responds optimally  by considering its demand preferences  and supply situations over a horizon. We  assume that the aggregator can anticipate the reaction of the prosumers. This price  oriented setup falls into the category of {bilevel optimization problems} \citep{colson2007overview}    and {Stackelberg games} \citep{von2010market}, where the lower level problem and the upper level problem  are the problems related to the prosumers and the aggregator, respectively.
\par
Bilevel optimization problems have extensively used to model and solve energy systems problems \citep{dempe2015bilevel}.
The initial work 
\citep{hobbs2000strategic} models  strategic offering of a dominant generating firm as a bilevel optimization problem, where at the upper level a generator firm maximizes its profit and at the lower level a system operator maximizes social welfare or minimizes total system cost. This problem is rewritten as a \ac{mpec} and solved by  a penalty interior point algorithm. More recent works (e.g., \citealt{zugno2013bilevel}) focus on the aggregator and prosumers problem. The state-of-the-art approach to solve these types of problems is to reformulate the bilevel optimization problem as a \ac{mip}. 
\par 
Both the \ac{mpec} and \ac{mip} based methods are computationally expensive. One of the main challenges to implement an \ac{rtm} is the computational efficiency. For an \ac{rtm} the time intervals are in the order of a few minutes  \citep{vlachos2013demand}.  Therefore, new computational tools are needed for the aggregator's real-time control over the prosumers and its participation  in the \ac{rtm}.     
\citealt{ghamkhari2016strategic} has addressed the computational efficiency of the dominant firm's strategic offering   by introducing  a convex relaxation for the bilevel optimization problem and has found a close to optimal solution. However, to the best of our knowledge, no study has been done on finding the global optimum of a bilevel optimization problem by solving a convex one in the field of prosumers integration in the  wholesale energy markets.
\par 
In this paper, we define the problem of economic optimization of an aggregator and its corresponding prosumers for participation in an \ac{rtm} over a time horizon as a bilevel optimization problem. The aggregator  represents the prosumers to participate in the wholesale market in a real-time scenario.  This problem, in general, is nonconvex \citep{luo1996mathematical}. We show that a subset of the set of global minimizers for the nonconvex  problem can be obtained as the solution of a certain convex optimization problem.  
The convex problem has two main advantages. On the one hand, a  convex formulation is attractive in real-time applications since the computation time is linear in the number of variables. On the other hand, off-the-shelf software packages can be used to solve the problem.
{In addition, replacing a bilevel optimization problem by a convex one is a key step toward  decentralized or distributed algorithms  \citep{Bertsekas/99}.} This work is a
continuation of the preliminary study by the authors \citep{2004.08612} which dealt with a simple static model for balancing markets.
\par 
The paper is organized as follows. In Section~\ref{sec:pf}, we define the aggregator and prosumers problems as a bilevel optimization problem. The results toward introducing  a convex optimization problem for the bilevel one comes in Section~\ref{sec:mr}.  Finally, the paper closes with the conclusions in Section~\ref{sec:con}.
\subsection*{Notation}
We denote the set of real numbers by $\R$, $n$-vectors by $\R^n$ and $m\times n$ matrices by $\R^{m\times n}$. Throughout the paper, the inequalities for vectors are meant entrywise. The $n$-vectors of ones is denoted by $\mathds{1}_n$. For vectors $x_i \in \R^{n_i}$ with $i=1,2,\dots,k$,  we write 
$\col (x_1, x_2, \dots, x_k) $
to denote the vector 
$\begin{bmatrix}
	x_1^T & x_2^T & \cdots& x_k^T
\end{bmatrix}^T$. The $m\times m$ identity matrix is denoted by $I_m$. For a matrix $M\in \R^{m \times n}$ and index sets $\alpha \subseteq \{1,2,\dots,m\}$, $\beta \subseteq \{1,2,\dots,n\}$, the notation $M_{\alpha\beta}$ denotes the matrix $ {\Big(
	M_{ij}
	\Big)}_{i \in \alpha, j \in \beta}$. If $\alpha=\{1,2,\dots,m\}$, then we write $M_{\bullet \beta}$ and if $\beta=\{1,2,\dots,n\}$, then we write $M_{\alpha \bullet}$. A symmetric matrix $M=M^T \in \R^{m\times m}$ is said to be positive semidefinite if  $x^TMx \ge 0$ for all $x \in \R^m$ and positive definite if $x^TMx >0$ for all $0 \neq x \in \R^m$. 
The symmetric square root of a positive definite matrix $M$ is denoted by $M^{\frac{1}{2}}$.
For a vector $v \in \R^n$, we write $\diag(v)$ for the diagonal matrix with diagonal entries $v_1, v_2, \dots, v_n$.
Let $f:\R^n\rightarrow\R$ and $S\subseteq\R^n$. Consider the optimization problem
\begin{subequations}\label{OP}
	\begin{alignat}{2}\mathrm{OP:}\qquad
		&\min_{x} \quad && f(x)\\
		& \mathrm{subject \ to} \quad && x \in S.
	\end{alignat}
\end{subequations}
We say that $\bar{x}$ is feasible for OP if $\bar{x} \in S$.
Also, we define the set of global optima for OP as
\begin{equation}
	\opt{\mathrm{OP}}=\{x^*\in S \mid f(x^*)\le f(x) \ \forall x \in S\}.
\end{equation}
\section{Problem Statement}\label{sec:pf}
In this section, we define a market model and platform for an aggregator and the  prosumers under its contract to participate in an \ac{rtm} with the grid, i.e., the wholesale market. Here, the role of the aggregator is to act as an intermediary agent between the prosumers and the grid to facilitate  the energy transactions. We consider the case where  each prosumer can generate energy through some \ac{res} with zero cost. Example of such energy sources are solar panels and wind turbines. Moreover, each prosumer's demand is elastic at each time-step. 
The aggregator goal  is to propose the prosumers with a personalized price to deal their surplus or shortage energy with the grid in an optimal way. 
The advantages of a personalized price over a unique price have been addressed in many recent research (see e.g., \citealt{tushar2014prioritizing}, \citealt{yang2018model}).
Next, we explain the problem setting and market structure in detail.
\subsection{Prosumer's Problem}
The main source of energy supply for a prosumer is its renewable energy units. Due to uncertain and uncontrollable nature  of \ac{res}, there might be a mismatch between supply and demand   at each time-step. Each prosumer has two options to cancel this mismatch. One is to trade  with the wholesale market through the aggregator. The other option is to use its demand elasticity. Therefore, the prosumer needs to find a trade-off between these two possible options for its optimal strategy. Before providing  a mathematical formulation   for the prosumer, we elaborate on demand elasticity.
\par
We say that the demand of each prosumer is elastic if:
\begin{enumerate}
	\item Each prosumer has a preference for its demand at each time-step.
	\item Altering the demand from its preferred value causes dissatisfaction for the prosumer. Here,  we model this dissatisfaction using a quadratic function.
	\item Each prosumer has a lower bound and an upper bound for its demand at each time-step.
	\item Total demand of each prosumer in a specific time period is constant.
\end{enumerate}
As explained before, the prosumer goal is to find a trade-off between two possible options to minimize its cost and maximize its comfort.
We model this problem as an optimization problem.  We define the set of prosumers by $\{1, 2, \dots, n\}$ and the set of time-steps by $\{1, 2, \dots, K\}$. 
Then, prosumer $i \in \{1, 2, \dots,n\}$  at time-step $k \in \{1, 2, \dots, K\}$ has three decision variables:
its demand $h_i(k)$, the energy it sells to (buy form) the grid $y_i^+(k)$ ($y_i^-(k)$). For the $i$th prosumer, we consider the following optimization problem: 
\begin{subequations}\label{prop}
	\begin{alignat}{2}\mathrm{PP}_i:\qquad
		&\min_{\substack{h_i(k),y_i^+(k),y_i^-(k) \\ \forall \ k \in \{1,2,\dots,K \} }} \quad && \sum_{k=1}^{K} \frac{1}{2}q_i(k)(h_i(k)-h_i^0(k))^2+x_i^-(k)y_i^-(k)-x_i^+(k)y_i^+(k) \label{propof}\\
		&\mathrm{subject \ to} \quad && y_i^+(k)-y_i^-(k)+h_i(k)=s_i(k) \quad \forall \ k \in \{1, \dots ,K\} \label{propeq} \\
		&&& y_i^+(k),y_i^-(k) \ge 0 \quad  \forall \ k \in \{1, \dots ,K\} \label{propnn} \\
		&&&\ubar h_i(k) \le h_i(k) \le \bar h_i(k) \quad \forall \ k \in \{1, \dots ,K\} \label{proplub} \\
		&&& \sum_{k=1}^K h_i(k) =h_i^{\mathrm{tot}} \label{proptot}
	\end{alignat}
\end{subequations}
where $x_i^+(k)$ ($x_i^-(k)$) is the proposed price by the aggregator to buy energy from (sell energy to) the prosumer at time-step $k$, $s_i(k) \ge 0$ is the generated energy by the prosumer at time-step $k$,  which assumed to be known, and $q_i(k) >0$ is the dissatisfaction parameter for the prosumer. Moreover, $h_i^0(k) \ge 0$, $\ubar h_i(k) \ge 0$ and $ \bar h_i(k) \ge 0$ are the preferred value, lower bound and upper bound for the demand $h_i(k)$, respectively. The parameter $h_i^{\mathrm{tot}}$ is the total demand for the prosumer over the period $k=1$ to $k=K$.
\par 
In \eqref{propof}, the first term models the dissatisfaction the prosumer experiences by changing its demand from the preferred value. The second term is the cost of buying energy from the grid through the aggregator and the third term is the revenue the prosumer can obtain by selling energy through the aggregator. The constraint \eqref{propeq} indicates that the total demand should be equal to  the total supply for each prosumer at each time-step. The constraints \eqref{propnn} and \eqref{proplub} specify the lower bound and upper bound for the decision variables.
Finally, \eqref{proptot} captures the assumption that the  total demand over a period is constant.
\begin{assumption}\label{asump1}
	The sum of preferred  values $h_i^0(k)$s over the period $k=1$ to $k=K$ is equal to $h_i^{\mathrm{tot}}$, i.e., 
	\begin{equation}
		\sum_{k=1}^K h_i^0(k)=h_i^{\mathrm{tot}}.
	\end{equation}
\end{assumption}
Form \eqref{propeq}, we can write $h_i(k)$ as
\begin{equation}
	h_i(k)=s_i(k)-(y_i^+(k)-y_i^-(k)).
\end{equation}
Thus, the variable $h_i(k)$ can be eliminated from the problem $\mathrm{PP}_i$ and we can rewrite it as the following optimization problem:
\begin{subequations}\label{propp}
	\begin{alignat}{2}\mathrm{PP}_i^\prime:\qquad
		&\min_{\substack{y_i^+(k),y_i^-(k) \\\forall \ k \in\{1,2,\dots,K\} }} \quad && \sum_{k=1}^{K} \frac{1}{2}q_i(k)(y_i^+(k)-y_i^-(k))^2+c_i(k) (y_i^+(k)-y_i^-(k)) \nonumber\\ &&& \qquad +x_i^-(k)y_i^-(k)-x_i^+(k)y_i^+(k) +\frac{1}{2}q_i(k)(h_i^0(k))^2 \label{proppof} \\
		&\mathrm{subject \ to} \quad && y_i^+(k),y_i^-(k) \ge 0 \quad \forall \ k \in \{1, \dots ,K\} \label{proppnn} \\
		&&&s_i(k) - \bar h_i(k)\le y_i^+( k)-y_i^-(k)\le s_i(k) - \ubar h_i(k) \quad \forall \ k \in \{1, \dots ,K\} \label{propplub} \\
		&&&- \sum_{k=1}^K (y_i^+(k)-y_i^-(k)) = h_i^{\mathrm{tot}}-\sum_{k=1}^K s_i(k) \label{propptot}
	\end{alignat}
\end{subequations}
where $c_i(k)=q_i(k)(h_i^0(k)-s_i(k))$. 
Moreover, since  the optimization problems $\mathrm{PP}_i^\prime$s are independent,  we can add them   and rewrite them in a vector form. To do so, we define the following vectors:
\begin{equation}
	\begin{aligned}
		&q=\col(q_1(1),q_1(2),\dots,q_n(K)),
		&&c=\col(c_1(1),c_1(2),\dots,c_n(K)), \\
		&h^0=\col(h_1^0(1),h_1^0(2),\dots,h_n^0(K)), 
		&&\ubar h=\col(\ubar h_1(1),\ubar h_1(2),\dots, \ubar h_n(K)),\\
		&\bar h=\col(\bar h_1(1),\bar h_1(2),\dots, \bar h_n(K)),
		&&s =\col(s_1(1),s_1(2),\dots,s_n(K)),\\
		&y^+=\col(y_1^+(1),y_1^+(2),\dots, y_n^+(K)),
		&&y^-=\col(y_1^-(1),y_1^-(2),\dots, y_n^-(K)),\\
		&x^+=\col(x_1^+(1),x_1^+(2),\dots, x_n^+(K)),
		&&x^-=\col(x_1^-(1),x_1^-(2),\dots, x_n^-(K)),\\
		&h^{\mathrm{tot}}=\col(h_1^{\mathrm{tot}},h_2^{\mathrm{tot}},\dots ,h_n^{\mathrm{tot}}).
	\end{aligned}
\end{equation}
Then, the vector form can be written as 
\begin{subequations}\label{props}
	\begin{alignat}{2}\mathrm{PP:}\qquad
		&\min_{y^+,y^-} \quad && \frac{1}{2}(y^+ - y^-)^TQ(y^+-y^-)+c^T (y^+-y^-) \nonumber\\ &&&  +({x^-})^Ty^- -({x^+})^Ty^+   \\
		&\mathrm{subject \ to} \quad && y^+,y^- \ge 0  \\
		&&&\ell \le y^+-y^-\le u \\
		&&& E (y^+-y^-) =d
	\end{alignat}
\end{subequations}	
where we have the following parameters:
\begin{gather}\label{parameters1}
	Q=\diag(q),  \quad c=Q(h^0-s), \quad \ell = s - \bar h,\\\label{parameters2}
	u=s - \ubar h, \quad E=-I_n \otimes \mathds{1}_K^T, \quad d= h^{\mathrm{tot}}- Es.
\end{gather}
Note that $\otimes$ denotes the Kronecker product.
\par
The prices $x_i^+(k)$ and $x_i^-(k)$ are proposed by the aggregator. In this work, the aggregator acts as a self-interest agent which has the ability to anticipate the reaction of the prosumers.   Therefore, knowing the reaction of the prosumers, the aggregator sets the prices to maximize its revenue as an intermediary player between the grid and the prosumers.
In the next  subsection, we elaborate on the aggregator's problem as a bilevel optimization problem.
\subsection{Aggregator's  Problem}
The aggregator receives two prices from the grid for each time-step. The price $p^+(k)$ is the price for selling energy to grid and the price $p^-(k)$ is the price for buying energy from the grid at $k$th time-step. Having these prices and the ability of the aggregator to anticipate the reaction of the prosumers  allow   the aggregator to propose prices $x_i^+(k)$ and $x_i^-(k)$ to the prosumers in an optimal way. The bilevel optimization below models this problem for the aggregator.
\begin{subequations}
	\begin{alignat}{2}\mathrm{AP:}\qquad
		&\max_{x^+,x^-,y^+,y^-} \quad &&  (p^+-x^+)^Ty^+-(p^--x^-)^Ty^- \label{apof}\\
		& \mathrm{subject \ to} \quad && x^+,x^- \ge 0 \\
		&&& (y^+,y^-) \in \opt{\mathrm{PP}} \quad 
	\end{alignat}
\end{subequations}
where
$p^+=\mathds{1}_n \otimes \col(p^+(1), p^+(2),\dots,p^+(K))$ and 	$p^-=\mathds{1}_n \otimes \col(p^-(1), p^-(2),\dots,p^-(K))$. The first term in \eqref{apof} corresponds to aggregator's revenue from selling energy to the grid. The second term models the aggregator's cost for buying energy from the grid.
\par 
In this paper, we consider a scenario where $p^+=-p^-=p$ and the aggregator proposes prices $x^+$ and $x^-$ such that $x^+=-x^-=x$. Therefore, we can rewrite the  optimization problems AP and PP based on the new decision variables $x$ and $y=y^+-y^-$ as the minimization problems BLP and LLP, respectively. 
\begin{subequations} \label{blp}
	\begin{alignat}{2} 
		\mathrm{BLP:}\qquad
		&\min_{x,y} && (x-p)^Ty \label{blpof}\\
		& \mathrm{subject \ to} \quad && x \ge 0 \label{blpx}\\
		&&& y=\opt{\mathrm{LLP}}. \label{blpll}
	\end{alignat}
\end{subequations}
Here the decision vector  $x\in \R^{m}$ is the proposed prices of  the aggregator  and the parameter vector  $p \in \R^{m}$ is the prices of selling to and buying from the grid.  The prosumers' reactions  $y$ to  the proposed prices  are the solution of  the optimization problem LLP.
\begin{subequations}\label{llp}
	\begin{alignat}{2}\mathrm{LLP:}\qquad
		&\min_y &&{\frac{1}{2}y^TQy+(c-x)^Ty} \label{llpof}\\
		& \mathrm{subject \ to} \quad && \ell \le y \le u  \label{llpinq}\\
		&&& Ey=d \label{llpeq}
	\end{alignat}
\end{subequations}
The vector $y \in \R^{m}$ is the decision variable for  LLP. The vectors and matrices $c,\ell, u \in \R^{m}$, $d\in \R^n$, $Q \in \R^{m\times m}$,  $E \in \R^{n\times m}$  are parameters for LLP as defined in \eqref{parameters1} and \eqref{parameters2}. Moreover,  $x\in \R^{m}$ is the decision variable for the aggregator and the prosumers  has no control over it. It should be noted that $m=nK$ and   $\rank E =n\le  m$. 
We assume that  there exists $\bar{y}$ which satisfies \eqref{llpinq}-\eqref{llpeq}.
Since  $Q=\diag(q)$ is positive definite, LLP is a strictly convex quadratic optimization  problem and hence has always a unique optimal solution, i.e., the set $\opt{\mathrm{LLP}}$ is a singleton. 
\par Bilevel optimization problems are in general nonconvex and have combinatorial nature. Many algorithms and approaches have been developed to solve different classes of bilevel problems. Recent surveys on bilevel optimization can be found in \citep{dempe2020bilevel} and \citep{luo1996mathematical}.  In contrast to existing methods that deal with rather more general bilevel optimization problems, our focus here  is to exploit the particular structure of \eqref{blp} in order to  introduce a convex  optimization  problem which has the same global optimum as the bilevel one.
The next section investigates the conditions under which the global optimal solution of the optimization \eqref{blp} can be  found by solving a convex  problem.
\section{Main Results}\label{sec:mr}
In this section, we will show that the set of global optima for  a specific   convex optimization problem is a subset of  the set of global optima for the optimization BLP, under some assumptions on the parameters of the problem.
Before dealing with the the optimization problem BLP, we consider two variations of this optimization problem.
First, we only consider  a lower bound on $y$ in \eqref{llpinq}. Later, we will consider an upper bound on $y$. Then, we come back to BLP \eqref{blp} to introduce a convex optimization problem  which can be used to find a subset of $\opt{\mathrm{BLP}}$.
Finally, we comment on the restrictions  of the proposed convex optimization for the \ac{rtm} platform.
\subsection{Lower Bound on $y$}
Consider the following bilevel optimization problem 
for which the decision variable $y$ has only a lower bound:
\begin{subequations}\label{blp1}
	\begin{alignat}{2}\mathrm{BLP1:}\qquad
		&\min_{x,y} \quad && \phi(x,y) \label{blp1of}\\
		& \mathrm{subject \ to} \quad && x \ge 0 \label{blp1x} \\ 
		&&& y=\opt{\mathrm{LLP1}} \label{blp1ll}
	\end{alignat}
\end{subequations}
where $\phi: \R^m\times \R^m \rightarrow\R$ is given by 
\begin{equation}
	\phi(x,y)=(x-p)^Ty,
\end{equation}
and LLP1 is as
\begin{subequations}\label{llp1}
	\begin{alignat}{2} \mathrm{LLP1:}\qquad
		&\min_y  \quad && \frac{1}{2}y^TRy+(c-x)^Ty \label{llp1of}\\
		&\mathrm{subject \ to} \quad && \ell \le y  \label{llp1inq}\\
		&&& Fy=d. \label{llp1eq}
	\end{alignat}
\end{subequations}
Here $R\in \R^{m \times m}$ is positive definite and  not necessarily diagonal and $F \in \R^{n\times m}$ has full row rank.  
Assume that   there exists $\bar y$  satisfying \eqref{llp1inq}-\eqref{llp1eq}
Since  LLP1 is a convex optimization problem, we can write the following necessary and sufficient KKT conditions to characterize  \eqref{blp1ll}:
\begin{gather}
	Ry+c-x+F^T\lambda-\mu=0,\label{KKTder}\\
	Fy=d, \label{KKTeq}\\
	0 \le \mu \bot y -\ell \ge 0 \label{KKTcoml}
\end{gather}
where $\mu\in \R^m$ and $\lambda \in \R^n$ are dual variables for the constraints \eqref{llp1inq} and  \eqref{llp1eq}, respectively.
The  dual variable $\lambda$ can be eliminated from KKT conditions \eqref{KKTder}-\eqref{KKTcoml}. 
First, we solve $y$ from \eqref{KKTder} as
\begin{equation}\label{eqy}
	y=R^{-1}(x+\mu-c-F^T\lambda),
\end{equation}
and then substitute $y$ in \eqref{KKTeq}:
\begin{equation}
	FR^{-1}(x+\mu-c-F^T\lambda)=d.
\end{equation}
Since $F$ has full row rank and $R$ is positive definite, $FR^{-1}F^T$ is nonsingular. Therefore, we obtain
\begin{equation}
	\lambda=(FR^{-1}F^T)^{-1}(FR^{-1}(x+\mu-c)-d).
\end{equation}
By substituting $\lambda$ in \eqref{eqy}, we can write  \eqref{KKTder}-\eqref{KKTcoml} as
\begin{gather} 
	y=M(x+\mu)+r, \label{KKTs1} \\ 
	0 \le \mu \bot y -\ell \ge 0  \label{KKTs2}
\end{gather}
where 
\begin{gather} \label{eqM}
	M=R^{-1}-R^{-1}F^T(FR^{-1}F^T)^{-1}FR^{-1},\\
	r=R^{-1}F^T(FR^{-1}F^T)^{-1}d-Mc.  \label{eqr}
\end{gather}
\begin{lemma}\label{lemM}
	The matrix $M$ is a positive semidefinite matrix.
\end{lemma}
\begin{proof}
Clearly $M$ is the Schur complement of 
\begin{equation}\label{schur}
	X=\begin{bmatrix}
		R^{-1} & R^{-1}F^T\\
		FR^{-1} & FR^{-1}F^T	
	\end{bmatrix}
\end{equation}
with respect to $ FR^{-1}F^T$. Since
$X=\begin{bmatrix}
	R^{-\frac{1}{2}} \\ F	R^{-\frac{1}{2}} 
\end{bmatrix}\begin{bmatrix}
	R^{-\frac{1}{2}} & 	R^{-\frac{1}{2}}F^T
\end{bmatrix}$, $X$ is positive semidefinite.
It follows from \citep[Theorem~1.12]{zhang2006schur} that $M$ is also positive semidefinite. 
  \end{proof}
As a result of \eqref{KKTs1} and \eqref{KKTs2}, BLP1 can be rewritten as the following  optimization problem:
\begin{subequations}\label{slp1}
	\begin{alignat}{2}\mathrm{SLP1:}\qquad
		&\min_{x,y} \quad &&  \phi(x,y)\\
		& \mathrm{subject \ to} \quad &&(x,y) \in S_1 \label{slp1c}
	\end{alignat}
\end{subequations}
where 
\begin{equation}
	S_1=\{(x,y)\ \mid  x \ge 0, \ y =M(x+\mu)+r, \ y \ge \ell, \ \mu^T(y-\ell)=0 \ \mathrm{for \ some} \  \mu \ge 0 \}.
\end{equation}
Since neither $\phi$ nor $S_1$ is convex, the optimization problem SLP1 is a nonconvex one. Nevertheless, a subset of the global minimizers of SLP1 can be captured by a convex optimization problem. 
\begin{theorem} \label{slp1tocvx1}
	Suppose that $\ell \le 0$. Consider the  optimization problem
	\begin{subequations}\label{cvx1}
		\begin{alignat}{2}\mathrm{CVX1:}\qquad
			&\min_{x,y} \quad &&  \phi(x,y)  \label{cvx1of}\\
			& \mathrm{subject \ to} \quad && (x,y) \in C_1 \label{cvx1c}
		\end{alignat}
	\end{subequations}
	where
	\begin{equation}
		C_1=\{(x,y) \mid x \ge 0, \ y=Mx+r, \ y \ge \ell \}.
	\end{equation} 
	Then, $\opt{\mathrm{CVX1}} \subseteq \opt{\mathrm{SLP1}}$. Furthermore, the optimization problem CVX1 is convex.
\end{theorem}
To prove the theorem above, we need some auxiliary results.  
The following lemma plays an essential role in the proof of Theorem~\ref{slp1tocvx1}.
\begin{lemma}\label{lemlb}
	Consider the sets $S_1^\prime=\{(x,y)\ \mid  x \ge 0, \ y =M(x+\mu)+r^\prime, \ y \ge \ell, \ \mu^T(y-\ell)=0 \ \mathrm{for \ some} \  \mu \ge 0 \}$ and $C_1^\prime=\{(x,y) \mid x \ge 0, \ y=Mx+r^\prime, \ y \ge \ell \}$ where $r^\prime \in \R^m$ is an arbitrary vector.
	Suppose that  $\ell \le 0$. Then, for any  $(\bar{x},\bar{y}) \in S_1^\prime$ there exists  $(\hat{x},\bar{y})\in C_1^\prime$    such that  $\phi(\hat{x},\bar{y})\le \phi(\bar{x},\bar{y})$. 
\end{lemma} 
\begin{proof}
Let $(\bar{x},\bar{y}) \in S_1^\prime$. 
Therefore, there exists $\bar \mu \ge 0$ such that
\begin{equation} \label{slp1cbar}
	\bar x \ge 0, \quad \bar y= M( \bar x+\bar \mu) +r^\prime, \quad \bar y\ge \ell, \quad \bar \mu^T (\bar y-\ell)=0. 
\end{equation} 
Take  $\hat x= \bar x +\bar  \mu$. Then, $\hat x \ge 0$, $M \hat x+r^\prime= \bar y \ge \ell$ and thus $(\hat x, \bar y) \in C_1^\prime$. 
\newline
Now, we define index  sets $\alpha \subseteq \{1, 2, \dots,m\}$ and  $\bar{\alpha}=\{1, 2,  \dots,m\} \setminus \alpha$ such that $\bar{\mu}_\an=0$ and $\bar{\mu}_\ab >0$.
Then, based on \eqref{slp1cbar}, we have 
\begin{gather}
	\hat x_\an = \bar x_\an, \quad \hat x_\ab > \bar x_\ab, \\
	\bar y_\an \ge \ell_\an, \quad \bar y_\ab = \ell_\ab.
\end{gather}
Since $\ell \le 0$, we have
\begin{gather}
	\hat x_\an=\bar x_\an \Longrightarrow (\hat x_\an-p_\an)^T\bar y_\an=(\bar x_\an-p_\an)^T \bar y_\an, \\
	\hat x_\ab > \bar  x_\ab \quad \text{and} \quad  y_\ab=\ell_\ab \le 0 \Longrightarrow (\hat x_\ab-p_\ab)^T\bar y_\ab\le (\bar x_\ab-p_\ab)^T \bar y_\ab,
\end{gather} 
and consequently $\phi(\hat x, \bar y)\le \phi(\bar x , \bar y)$.
  \end{proof}
Now, we are ready to prove Theorem~\ref{slp1tocvx1}.
\begin{proof}[Proof of Theorem~\ref{slp1tocvx1}.]
Let $(x^*,y^*) \in \opt{\mathrm{CVX1}}$ and $(\bar x, \bar y) \in S_1$. 
It follows form Lemma~\ref{lemlb}, with the choices $r^\prime = r$, $S_1^\prime=S_1$ and  $C_1^\prime=C_1$, that 
there exists $(\hat x, \bar y) \in C_1$ such that
\begin{equation}
	\phi(\hat x,\bar  y) \le  \phi(\bar x, \bar  y).
\end{equation}
Therefore, we see that 
\begin{equation}
	\phi(x^*,y^*) \le \phi(\hat x,\bar y) \le  \phi(\bar x, \bar  y).  
\end{equation}
This means that $\phi(x^*, y^*) \le \phi(\bar x, \bar y)$ for all $(\bar x, \bar y) \in S_1$.
Since $C_1 \subseteq S_1$, $(x^*,y^*) \in S_1$ and hence $(x^*,y^*) \in \opt{\mathrm{SLP1}}$. Thus, we can conclude  that  $\opt{\mathrm{CVX1}} \subseteq \opt{\mathrm{SLP1}}$. 
Note that, $C_1$ is a polyhedron. Furthermore, $\phi$ is convex since $M$ is positive semidefinite due to Lemma~\ref{lemM}. Therefore, CVX1 is a convex optimization problem. 
  \end{proof}
\subsection{Upper Bound on $y$}
Next, we consider the following bilevel problem for which the decision variable $y$ has only an upper bound:  
\begin{subequations}\label{blp2}
	\begin{alignat}{2}\mathrm{BLP2:}\qquad
		&\min_{x,y}\quad &&  \phi(x,y) \label{blp2of}\\
		& \mathrm{subject \ to} \quad && x \ge 0  \label{bl2l} \\ 
		&&& y=\opt{\mathrm{LLP2}} \label{blp2ll}
	\end{alignat}
\end{subequations}
where LLP2 is given by
\begin{subequations}\label{llp2}
	\begin{alignat}{2} \mathrm{LLP2:}\qquad
		&\min_y  \quad && \frac{1}{2}y^TRy+(c-x)^Ty \label{llp2of}\\
		&\mathrm{subject \ to} \quad && y \le u \label{llp2inq}\\
		&&& Fy=d. \label{llp2eq}
	\end{alignat}
\end{subequations}
Assume that there exists $\bar y$ which  satisfies \eqref{llp2inq}-\eqref{llp2eq}. 
In a procedure similar to what we had for BLP1, we can show that BLP2 can be written as the following optimization problem using the KKT conditions:
\begin{subequations}\label{slp2}
	\begin{alignat}{2}\mathrm{SLP2:}\qquad
		&\min_{x,y} \quad &&  \phi(x,y) \label{slp2of}\\
		& \mathrm{subject \ to} \quad &&(x,y) \in S_2 \label{slp2c}
	\end{alignat}
\end{subequations}
where 
\begin{equation}
	S_2=\{(x,y)\ \mid  x \ge 0, \ y =M(x-\nu)+r, \ y \le u, \ \nu^T(u-y)=0 \ \mathrm{for \ some} \  \nu \ge 0 \},
\end{equation}
and  $M$ and $r$ are as in \eqref{eqM} and \eqref{eqr}.  Note that $\nu$ is the dual variable for the constraint \eqref{llp2inq}.
\par Again here we want to find sufficient conditions such that at least a global optimum of SLP2 can be found using a convex optimization problem. The main assumption here is on  the structure of  $M$. The following definitions elaborate on this specific structure.
\begin{definition} A matrix $N \in \R^{k\times k}$ is called
	\begin{itemize}
		\item a Z-matrix if its  off-diagonal entries are nonpositive.
		\item an M-matrix if it is an Z-matrix and the real part of its eigenvalues are nonnegative.  
	\end{itemize}
\end{definition}
\begin{remark}
	In particular, a positive semidefinite matrix is an M-matrix if  its off-diagonal entries are nonpositive. 
\end{remark}
Now, we can state the main result concerning the the optimization problem SLP2.
\begin{theorem}\label{slp2tocvx2}
	Suppose that $M$ is an  M-matrix, $u\ge0$ and $u> r$. Consider the following optimization problem:
	\begin{subequations}\label{cvx2}
		\begin{alignat}{2}\mathrm{CVX2:}\qquad
			&\min_{x,y} \quad &&  \phi(x,y)  \label{cvx2of}\\
			& \mathrm{subject \ to} \quad && (x,y) \in C_2 \label{cvx2c}
		\end{alignat}
	\end{subequations}
	where 
	\begin{equation}
		C_2=\{(x,y) \mid x \ge 0, \ y=Mx+r, \ y \le u \}.
	\end{equation}
	Then, $\opt{\mathrm{CVX2}} \subseteq \opt{\mathrm{SLP2}}$. Furthermore, the optimization problem CVX2 is convex.
\end{theorem}
The following results are needed to prove Theorem~\ref{slp2tocvx2}. 
\begin{lemma}\label{lemub}
	Consider the following optimization problem:
	\begin{subequations} \label{slp2p}
		\begin{alignat}{2}\mathrm{SLP2^\prime:}\qquad
			&\min_{x,y} \quad &&  \phi(x,y) \label{slp2pof}\\
			& \mathrm{subject \ to} \quad && (x,y) \in S_2^\prime \label{slp2pc}
		\end{alignat}
	\end{subequations}
	where 
	\begin{equation}
		S_2^\prime=\{(x,y) \mid x \ge 0, \ y =M(x-\nu)+r, \ y  \le u,  \ \nu^T(u-y)=0, \ \nu^Tx=0  \ \mathrm{for \ some} \  \nu \ge 0  \}.
	\end{equation}
	Suppose that $u\ge0$. 
	Then, for any  $(\bar{x},\bar{y}) \in S_2$ there exists  $(\hat{x},\bar{y})\in S_2^\prime$    such that  $\phi(\hat{x},\bar{y})\le \phi(\bar{x},\bar{y})$. 
\end{lemma}
\begin{proof}
Let $(\bar x, \bar y) \in S_2$. Therefore, there exists $ \bar \nu \ge 0$ such that
\begin{equation}\label{lp2cbar}
	\bar x \ge 0, \quad \bar y = M(\bar x-\bar \nu) +r \le u, \quad \bar\nu^T(u-\bar y)=0.
\end{equation}
We define index sets $\alpha\subseteq\{1, 2, \dots,m\}$ and $\ab=\{1, 2, \dots,m\}\setminus\an$ such that $\bar \nu_\an=0$ and $\bar \nu_\ab>0$. Then, \eqref{lp2cbar} can be rewritten based on $\an$ and $\ab$ as 
\begin{gather}
	\bar x_\an \ge 0 , \quad \bar x_\ab \ge 0, \\
	\bar y_\an=M_{\an\an}\bar x_\an+M_{\an\ab}(\bar x_\ab-\bar \nu_\ab)+r_\an \le u_\an, \quad \bar \nu_\an=0,\\
	\bar y_\ab=M_{\ab\an}\bar x_\an+M_{\ab\ab}(\bar x_\ab-\bar \nu_\ab)+r_\ab = u_\ab, \quad \bar \nu_\ab>0. 
\end{gather}	
We can choose $\hat x,\hat \nu$ based on $\bar x, \bar \nu$ as 
\begin{gather}
	\hat x_\an=\bar x_\an \ge 0, \quad \hat \nu_\an=\bar \nu_\an=0,\\
	\begin{bmatrix}
		\hat x_\ab \\ \hat \nu_\ab
	\end{bmatrix}=
	\begin{cases}
		\begin{bmatrix}
			\bar x_\ab-\bar \nu_\ab \\
			0
		\end{bmatrix}
		& \text{if }\quad \bar x_\ab-\bar \nu_\ab \ge 0, \\[30pt]
		\begin{bmatrix}
			0 \\ 
			\bar \nu_\ab - \bar x_\ab
		\end{bmatrix}
		& \text{if }\quad \bar x_\ab-\bar \nu_\ab < 0. 
	\end{cases} 
\end{gather}
Note that these choices imply that 
\begin{equation}
	0 \le \hat x_\ab < \bar x_\ab\quad \text{and} \quad  \hat \nu_\ab \ge 0.
\end{equation}
Then, we have $\hat x \ge 0$, $M(\hat x -\hat \nu)+r=\bar y \le u$, $\hat \nu^T(u-\bar y)=0$, $\hat \nu^T\hat x=0$ and $\hat{\nu} \ge 0$. Thus, $(\hat x, \bar y) \in S_2^\prime$.
Since $u \ge 0$, we have the following implications:
\begin{gather}
	\hat x_\an =\bar x_\an \Longrightarrow (\hat x_\an-p_\an)^T\bar y_\an =(\bar x_\an-p_\an)^T\bar y_\an, \\
	\hat x_\ab < \bar x_\ab \quad \text{and} \quad  \bar y_\ab = u_\ab \ge 0 \Longrightarrow (\hat x_\ab-p_\ab)^T\bar y_\ab \le (\bar x_\ab-p_\ab)^T\bar y_\ab
\end{gather}
which conclude that $\phi(\hat{x}, \bar{y})\le \phi(\bar{x},\bar{y})$. 
  \end{proof}
The set $S_2^\prime$ is a nonconvex set due to complementarity terms.
In what follows, we will show that under some conditions on $M,u$ and $r$, the set $S_2^\prime$ is equal to the  polyhedral set $C_2$ in CVX2.
\par
Let $(\bar x, \bar y) \in S_2^\prime$. Therefore, there exists $ \bar \nu \ge 0$ such that 
\begin{equation}\label{s2p}
	\bar{x} \ge 0, \quad \bar{y}=M(\bar{x}-\bar{\nu})+r \le u, \quad \bar  \nu^T(u-\bar y)=0, \quad \bar \nu^T \bar x=0.
\end{equation} 
We define  index sets $\alpha\subseteq\{1,2,\dots,m\}$ and $\ab=\{1,2,\dots,m\}\setminus\an$ such that $\bar \nu_\an=0$ and $\bar \nu_\ab>0$. Then, the following implications follow from \eqref{s2p}:
\begin{gather}
	\bar \nu_\an=0 \Longrightarrow \bar x_\an \ge 0\quad \text{and} \quad \bar y_\an=M_{\an\an}\bar x_\an-M_{\an\ab}\bar \nu_\ab+r_\an \le u_\an, \\ 
	\bar \nu_\ab > 0 \Longrightarrow \bar x_\ab = 0\quad \text{and} \quad \bar y_\ab=M_{\ab\an} \bar x_\an-M_{\ab\ab}\bar \nu_\ab+r_\ab = u_\ab.
\end{gather}
Consequently, we can write the following system of (in)equalities for $S_2^\prime$:
\begin{equation}\label{system*}
	\begin{aligned}
		\bar y_\an=M_{\an\an}\bar x_\an-M_{\an\ab}\bar \nu_\ab+r_\an \le u_\an, \quad  \bar x_\an \ge 0, \\ 
		\bar y_\ab=M_{\ab\an}\bar x_\an-M_{\ab\ab}\bar \nu_\ab+r_\ab = u_\ab,  \quad \bar \nu_\ab > 0.
	\end{aligned}
\end{equation}
Next, we elaborate on some tools to work with   linear (in)equalities.
\begin{lemma}[{\citealt[Theorem~3.8.3]{cottle2009linear}}]\label{lemcp}
	Let $N\in \R^{k\times k}$ be a positive semidefinite matrix. Then,
	for every index set $\beta \subseteq\{1,2,\dots,k\}$, the inequality  system  
	\begin{equation*}
		N_{\beta\beta}\xi \le 0, \qquad \xi \le 0
	\end{equation*}
	has a nonzero solution $\xi$.
\end{lemma}
\begin{lemma}[{\citealt[Theorem 2.7.9]{cottle2009linear}}]\label{lemfarkas}
	Let $A\in \R^{k_1\times k_2}$ and $b\in \R^{k_1}$ be given. Exactly one of the following statements holds:
	\begin{enumerate}[label=\roman*)]
		\item  There exists $w \ge 0$ such that $Aw \le b$.
		\item There exists $z \ge 0$ such that  $A^Tz \ge 0 \quad \text{and} \quad   b^Tz <0$.
	\end{enumerate}
\end{lemma}
\par The following lemma provides sufficient conditions for \eqref{system*} and $S_2^\prime$ to be feasible for the index set $\alpha=\{1,2,\dots,m\}$ and infeasible $\alpha \neq \{1,2,\dots,m\}$.
\begin{lemma}\label{lemrnu}
	Suppose that $M$ is an   {\normalfont M}-matrix and $u > r$. Then, $S_2^\prime=C_2$, i.e.,  the system \eqref{system*} is only feasible when $\nu=0$.
\end{lemma}	
\begin{proof}
Let $\alpha=\{1,2,\dots,m\}$. The alternative system for  \eqref{system*} is as
\begin{equation}
	Mz \ge 0, \quad (u-r)^Tz <0, \quad z \ge 0
\end{equation}
which clearly have no solutions since $u-r$ is positive. Therefore, it follows from Lemma~\ref{lemfarkas} that  \eqref{system*}  has at least one solution.
\par Now let $\alpha \neq \{1,2,\dots,m\}$ which means $\ab$ is a nonempty set. The alternative system for \eqref{system*} in this case is given by 
\begin{equation} \label{system**}
	\begin{bmatrix}
		M_{\an\an} & M_{\an\ab} \\ -M_{\ab\an} & -M_{\ab\ab}
	\end{bmatrix}
	\begin{bmatrix}
		z_\an \\ z_\ab
	\end{bmatrix}\ge 0, \quad \begin{bmatrix}
		u_\an-r_\an\\ u_\ab-r_\ab
	\end{bmatrix}^T\begin{bmatrix}
		z_\an \\ z_\ab 
	\end{bmatrix} <0 , \quad 
	z_\an \ge 0. 
\end{equation}
We argue that this system has at least one solution. To see this, take $z_\an=0$.  This leads to 
\begin{equation} \label{system**red1}
	\begin{bmatrix}
		M_{\an\ab} \\ -M_{\ab\ab}
	\end{bmatrix} z_\ab \ge 0, \quad (u_\ab-r_\ab)^Tz_\ab <0.
\end{equation}
Since $u > r$ and hence $u_\ab > r_\ab$, \eqref{system**red1} has a solution if and only if $z_\ab \le 0$ and $z_\ab \neq 0$. Consequently, \eqref{system**red1} has a solution if and only if the following has a solution:
\begin{equation} \label{system**red2}
	\begin{bmatrix}
		M_{\an\ab} \\ -M_{\ab\ab}
	\end{bmatrix} z_\ab \ge 0, \quad z_\ab \le 0, \quad z_\ab \neq 0.
\end{equation}
Since $M$ is  positive semidefinite due to Lemma~\ref{lemM}, it follows form  Lemma~\ref{lemcp} that  there exists a nonzero $z_\ab$ such that $-M_{\ab\ab}z_\ab \ge 0$ and $z_\ab \le 0$. As $M$ is an M-matrix, $M_{\an\ab}$ is nonpositive. Therefore, we have that   $M_{\an\ab}z_\ab \ge 0$ which concludes  that  \eqref{system**} has at least one solution and hence \eqref{system*} is infeasible for $\an\neq\{1,2,\dots,m\}$.
  \end{proof}
After these preparations, we are in a position to prove Theorem~\ref{slp2tocvx2}.
\begin{proof}[Proof of Theorem~\ref{slp2tocvx2}.]
Based on Lemma~\ref{lemrnu}, $S_2^\prime=C_2$ and hence optimization problems $\mathrm{SLP2^\prime}$ and CVX1  are the same. Therefore, $\opt{\mathrm{CVX2}}=\opt{\mathrm{SLP_2^\prime}}$.
Thus, it suffices to prove $\opt{\mathrm{SLP2^\prime}} \subseteq \opt{\mathrm{SLP2}}$.
\\
Let $(x^*,y^*) \in \opt{\mathrm{SLP_2^\prime}}$. Then, 
\begin{equation}
	\phi(x^*,y^*) \le \phi (x,y) \quad \forall  \ (x,y) \in S_2^\prime.
\end{equation}
Let $(\bar x, \bar y) \in S_2$. Then, it follows from Lemma~\ref{lemub} that  there exists $(\hat x, \bar y) \in S_2^\prime$ such that
\begin{equation}
	\phi(\hat x,\bar y) \le  \phi(\bar x, \bar  y).
\end{equation}
Consequently, we obtain
\begin{equation}
	\phi(x^*,y^*) \le \phi(\hat x,\bar y) \le  \phi(\bar x, \bar  y).  
\end{equation}
This means that $\phi(x^*,y^*) \le \phi(\bar x,\bar y)$ for all $(\bar x, \bar y) \in S_2$.
It is clear that  $S_2^\prime \subseteq S_2$. Therefore, $(x^*,y^*) \in S_2$ and hence  $(x^*,y^*) \in \opt{\mathrm{SLP2}}$. Thus, we can conclude $\opt{\mathrm{CVX2}}=\opt{\mathrm{SLP2^\prime}} \subseteq \opt{\mathrm{SLP2}}$. 
Furthermore, $C_2$ is a polyhedron and hence convex. Also, $\phi$ is convex on $C_2$ since $M$ is positive semidefinite due to Lemma~\ref{lemM}. Therefore, CVX2 is a convex optimization problem. 
  \end{proof}
\subsubsection{On $M$ being an M-matrix:}
One of the assumptions in Theorem~\ref{slp2tocvx2} is on the  structure of the matrix $M$. 
We know that $M$ is the Schur complement of matrix $X$ as mentioned in the proof of Lemma~\ref{lemM} and \eqref{schur}. Here, we discuss when  $M$ as in \eqref{eqM} is an M-matrix.
The following lemma shows that the Schur complement of an M-matrix is also an M-matrix.
\begin{lemma}[{\citealt[Theorem~5.13]{fiedler2008special}}]\label{lemMSchur}
	Suppose that  $N$ is an {\normalfont M}-matrix. Then, Schur complement of $N$ with respect to a positive definite submatrix of $N$ is also an {\normalfont M}-matrix.
\end{lemma}
The following theorem  provides sufficient conditions for $M$ to be an M-matrix.
\begin{theorem} \label{theoMmatirx}
	Suppose that $F$ is  a nonpositive matrix and rows of $F$ are orthogonal, i.e., $F_{ i \bullet } (F_{j\bullet })^T=0$ for all $i \neq j$. Then, $M$ is an M-matrix  if $R$ is diagonal. 
\end{theorem}
\begin{proof}
The matrix $M$ is the Schur complement of matrix $X$ given by \eqref{schur}. Since  $F$ is a matrix with  orthogonal rows and $R$ is positive definite and diagonal, $FR^{-1}F^T$ and hence $(FR^{-1}F^T)^{-1}$  are also positive definite diagonal matrices. Moreover,   $FR^{-1}$ is nonpositive which makes $X$ an M-matrix. Consequently,  $M$ is also an M-matrix based on Lemma~\ref{lemMSchur}.
  \end{proof}
\subsection{General Case}
Here, we consider the main problem, i.e. the optimization problem \eqref{blp}. We rewrite BLP by characterizing lower level problem based on KKT conditions as
\begin{subequations}\label{slp}
	\begin{alignat}{2}\mathrm{SLP:}\qquad
		&\min_{x,y} \quad &&  \phi(x,y) \\
		& \mathrm{subject \ to} \quad &&(x,y) \in S
	\end{alignat}
\end{subequations}
where 
\begin{multline}
	S=\{(x,y) \mid x \ge 0, \ y=M(x+\mu-\nu)+r, \ \ell \le y \le u, \ \mu^T(y-\ell)=0, \ \nu^T(u-y)=0\\ \ \mathrm{for \ some} \ \mu \ge 0 \ \mathrm{and} \ \nu\ge 0 \}.
\end{multline}
Note that $\mu$ and  $\nu$ are dual variables for $\ell \le y$ and $y \le u$ in \eqref{llpinq}, respectively. Moreover, the dual variable  for the constraint \eqref{llpeq} has been eliminated  from KKT conditions in  a similar way to \eqref{KKTder}-\eqref{KKTs2}. Also, $M$ and $r$ are as in \eqref{eqM} and \eqref{eqr}, respectively with $R=Q$ and $F=E$. As a result, $M$ is an M-matrix based on Theorem~\ref{theoMmatirx}.
\par
The  theorem below indicates that  there exists a convex optimization problem which can capture a subset of the set of global optima for SLP. 
\begin{theorem}\label{slptocvx}
	Consider the following optimization problem:
	\begin{subequations}\label{cvx}
		\begin{alignat}{2}\mathrm{CVX:}\qquad
			&\min_{x,y} \quad &&  \phi(x,y)  \label{cvxof}\\
			& \mathrm{subject \ to} \quad && (x,y) \in C \label{cvxc}
		\end{alignat}
	\end{subequations}
	where
	\begin{equation}
		C=\{(x,y) \mid x \ge 0, \ y=Mx+r, \ \ell \le y \le u \}. 
	\end{equation} 
	Suppose that $\ell \le 0$, $u \ge 0$, and $u > r $. 
	Then, $\opt{\mathrm{CVX}} \subseteq \opt{\mathrm{SLP}}$. Furthermore, the optimization problem CVX is convex.
\end{theorem}
\begin{proof}
Let $(\bar x, \bar y) \in S$. Therefore, there exist $\bar \mu \ge 0$ and $\bar \nu \ge 0$ such that 
\begin{equation} \label{slpcbar}
	\bar x \ge 0, \quad  \bar y=M(\bar x+\bar \mu-\bar \nu)+r, \quad \ell \le \bar y \le u, \quad \bar \mu^T(\bar y-\ell)=0, \quad \bar \nu^T(u-\bar y)=0.
\end{equation} 
Let $\bar y=M(\bar x+\bar \mu) +r^\prime$ where $r^\prime =r-M\bar v$.
As a result, $(\bar x, \bar y) \in S_1^\prime$.
From Lemma~\ref{lemlb}, there exists $(\hat x_1, \bar y) \in C_1^\prime$ such that 
\begin{equation}\label{eq:in1}
	\phi (\hat x_1,\bar y) \le \phi (\bar x,\bar y).
\end{equation}
Moreover, as $\bar y$ satisfies $\bar y \le u, \ \bar \nu^T(u-\bar y)=0$ for  $\bar \nu \ge 0$ and $r^\prime=r-M\bar \nu$, we can conclude
\begin{equation}
	(\hat x_1, \bar y) \in S_2 \cap \{(x,y) \mid y  \ge \ell \}.
\end{equation}
Thus, $(\hat x_1, \bar y) \in S_2$.
Due to Lemma~\ref{lemub} and Lemma~\ref{lemrnu}, there exists $(\hat x_2,\bar y) \in C_2$ such that 
\begin{equation}\label{eq:in2}
	\phi (\hat x_2,\bar y) \le \phi (\hat x_1,\bar y).
\end{equation}
Furthermore, since $\bar y \ge \ell$, $(\hat x_2,\bar y) \in C$. 
Then,  \eqref{eq:in1} and \eqref{eq:in2} imply that 
\begin{equation} \label{eq:in12}
	\phi (\hat x_2,\bar y) \le \phi (\bar x,\bar y) \quad \forall \ (\bar x, \bar y) \in S.
\end{equation}
Now, let $(x^*,y^*) \in \opt{\mathrm{CVX}}$. That is,
\begin{equation}
	\phi(x^*,y^*) \le \phi(x,y) \quad \forall \ (x,y) \in C.
\end{equation}
Since $(\hat x_2, \bar y) \in C$, \eqref{eq:in12} implies that
\begin{equation}
	\phi(x^*,y^*) \le \phi(\bar x, \bar  y) \quad \forall \ (\bar x ,\bar y) \in S.
\end{equation}
Since $C \subseteq S$, $(x^*,y^*) \in S$.  Consequently, $(x^*,y^*) \in \opt{\mathrm{SLP}}$ and hence $\opt{\mathrm{CVX}}\subseteq \opt{\mathrm{SLP}}$. 
Furthermore, $C$ is a polyhedral  convex set. Also, $\phi$ is convex on $C$ since $M$ is positive semidefinite due to Lemma~\ref{lemM}. Therefore, CVX is a convex optimization problem.
  \end{proof}
\subsection{Interpretation of Theorem~\ref{slptocvx}  for  \ac{rtm}} 
Theorem~\ref{slptocvx} has certain hypotheses on the parameters $\ell$, $u$, and $r$. Here, we discuss the implications of these hypotheses for the proposed \ac{rtm} platform.
It follows from Theorem~\ref{slptocvx} that the  vectors  $\ell=s-\bar{h}$ and $u=s-\ubar{h}$ should be nonpositive and nonnegative, respectively. The vector $s$ is the  generated energy by \ac{res} of the prosumers, 
the vector $\ubar{h}$ is the lower bound of the prosumers' demand  and 
the vector $\bar{h}$ is the upper bound for their  demand. 
To have  $\ell \le 0$ and $u \ge 0$, the aggregator should ask the prosumers to set the upper bound of their demands $\bar{h}$ greater than or equal to their \ac{res}' capacity and also the lower bound for their demands $\ubar{h}$  equal to zero.
Moreover, Theorem~\ref{slptocvx} states that $u$ should be  strictly greater than $ r $ where 
$r$ is defined as in \eqref{eqr} with $R=Q$ and $F=E$. Considering Assumption~\ref{asump1}, we can show that 
\begin{equation}
	u > r \iff h^0 >\ubar{h}.
\end{equation}  
Therefore, to have $u > r $, the aggregator should ask the prosumers to set their preferred values $h^0$ greater than  the lower bound for their demand $\ubar{h}$, i.e., $h^0 >0$. 
\section{Conclusions}\label{sec:con}
The problem of participation of the prosumers in the wholesale market through the aggregator has been widely studied in the literature. To represent the intrinsic hierarchy of this problem, we developed a market platform based on a bilevel optimization problem. Bilevel optimization are generally highly nonconvex and current approaches to deal with these problems are computationally expensive. To implement this market platform in real-time, we proposed a specific convex optimization problem and showed that  each global minimizer of this convex problem are also a global minimizer for the original bilevel problem under some assumptions on the parameters.
\par 
While the proposed convex approach can reduce the computational time significantly in contrast to the state-of-the-art methods (e.g., \ac{mip}), the assumption  that the aggregator has a centralized control over the prosumers may limit the applicability of the proposed method to large scale networks. An interesting important  area of future research  could be design of a decentralized or distributed control mechanism using the  convex problem  to tackle this issue.

\bibliographystyle{unsrtnat}
\bibliography{mybib}  

\end{document}